\documentclass[11pt]{amsart}

\usepackage{amsthm}
\usepackage{amssymb}
\usepackage{latexsym}
\usepackage{euscript}
\usepackage{amsmath} 
\usepackage{amstext}
\usepackage{amsgen}
\usepackage{amsbsy}
\usepackage{amsopn}
\usepackage{amsfonts}
\usepackage{hyperref}
\usepackage{graphicx}
\usepackage{bbm}
\usepackage{tikz}
\usepackage{tikz-cd}
\usepackage{enumerate}
\usepackage[shortlabels]{enumitem}

\newtheorem{theorem}{Theorem}[section]

\newtheorem{corollary}[theorem]{Corollary}
\newtheorem{lemma}[theorem]{Lemma}

\theoremstyle{definition}
\newtheorem{definition}[theorem]{Definition}
\newtheorem{example}[theorem]{Example}

\newtheorem{remark}[theorem]{Remark}
\newtheorem{remarks}[theorem]{Remarks}

\newcommand{\defn}[1]{\emph{#1}}

\newcommand{\ie}{i.e.\ }

\newcommand{\Z}{\mathbb{Z}}

\newcommand{\R}{\mathbb{R}}
\newcommand{\C}{\mathbb{C}}

\newcommand{\HH}{\mathbbm{H}}
\newcommand{\PP}{\mathbb{P}}

\newcommand{\cat}[1]{\mathcal{#1}}
\renewcommand{\mod}[1]{{#1}\mathrm{-mod}}
\newcommand{\alg}[1]{\mathrm{#1}}

\newcommand{\der}[2]{\cat{D}^{#1}({#2})}

\newcommand{\id}{\mathrm{id}}

\newcommand{\Mor}[3]{{\mathrm{Hom}}_{#1}\!\left(#2,#3\right)}

\newcommand{\Ext}[4]{{\mathrm{Ext}_{#1}^{#2}}\!\left(#3,#4\right)}
\newcommand{\End}[2]{\mathrm{End}_{#1}\!\left(#2\right)}
\newcommand{\Filt}[1]{\mathrm{Filt}\!\left(#1\right)}
\newcommand{\extcl}[1]{\Filt{#1}}
\newcommand{\length}[1]{\ell(#1)}

\newcommand{\epic}{\twoheadrightarrow}
\newcommand{\monic}{\hookrightarrow}

\newcommand{\stan}[1]{{\Delta_{#1}}}
\newcommand{\costan}[1]{{\nabla_{\! #1}}}
\newcommand{\pstan}[1]{{\underline{\Delta}_{#1}}}

\newcommand{\kvect}{\mathrm{Vect}_{k}}

\newcommand{\gldim}[1]{\mathrm{gldim}(#1)}

\renewcommand{\phi}{\varphi}

\hypersetup{
    colorlinks=true,
    linkcolor=red}

\title[Highest weight categories and stability conditions]{Highest weight categories \\ and stability conditions}
\author{Alessio Cipriani and Jon Woolf}
\address{Dipartimento di Informatica - Settore di Matematica, Universit\`a degli Studi di Verona, Strada le Grazie 15 - Ca' Vignal, I-37134 Verona, Italy}
\email{alessio.cipriani@univr.it}
\address{Jon Woolf, Dept. of Mathematical Sciences, University of Liverpool, L69 7ZL, U.K.}
\email{jonwoolf@liverpool.ac.uk}
\begin{document}

\begin{abstract}
Highest weight categories are an abstraction of the representation theory of semisimple Lie algebras introduced by Cline, Parshall and Scott in the late 1980s.  There are by now many characterisations of when an abelian category is highest weight, but most are hard to verify in practice. We present two new criteria --- one numerical in terms of the Grothendieck group, and one in terms of Bridegland stability conditions --- which are easier to verify. The stability criterion naturally  generalises to a characterisation of properly stratified categories. The numerical criterion implies a criterion of Green and Schroll  for when modules over a monomial algebra are highest weight.
\end{abstract}

\subjclass  {16G10 (primary), 16G20 (secondary)}
\keywords{Highest-weight category, Bridgeland stability conditions}

\maketitle

\section{Introduction}

In their seminal paper \cite{Cline1988}, Cline, Parshall and Scott observed that there are structural resemblances between certain module categories and categories of perverse sheaves, and moreover that these can be explained in terms of the properties of `standard objects', which play a similar role to Verma modules in Lie theory. These standard objects, and dual costandard ones, are defined whenever we choose an ordering of the simple objects of a Deligne-finite category, \ie a category equivalent to finitely-generated modules over a finite-dimensional $k$-algebra. They are in bijection with the simple, and therefore also the indecomposable projective, objects. In a highest-weight category each indecomposable projective has a `good' filtration by `larger' standard objects; Bernstein--Gelfand--Gelfand reciprocity relates the multiplicities in this to those in the Jordan--H\"older filtrations of costandard objects.

Any Deligne-finite category $\cat{A}$ is equivalent to the category of finitely-generated modules over a finite-dimensional algebra. Cline, Parshall and Scott \cite[Thm 3.6]{Cline1988} showed that $\cat{A}$ is a highest-weight category precisely when this algebra is quasi-hereditary. They also showed that quasi-hereditary algebras have finite global dimension; Dlab and Ringel \cite{MR0987824} showed the converse is false, but that every algebra of global dimension at most $2$ is quasi-hereditary. 

The structural commonality between certain categories of perverse sheaves and modules over a quasi-hereditary algebra identified by Cline, Parshall and Scott is the existence of a particularly simple iterated recollement. The central example is the category of perverse sheaves on a complex variety with a stratification by affine linear spaces; this is highest-weight by Mirollo--Vilonen \cite[Thm 1.3]{MR0925719} or Beilinson--Ginzburg--Soergel \cite[Thm 3.3.1]{MR1322847}. Here, the highest-weight structure is intimately related to mixed geometry, Koszul duality and Kazhdan--Lusztig theory \cite{MR1322847,MR1245719}. The distinguishing features of the iterated recollement of a highest-weight category are first that each factor is the category of vector spaces, and second that the recollement is homological, equivalently extends to a recollement of the bounded derived category --- see Krause \cite{MR3742477}. 

The standard objects in a highest-weight category form a full exceptional sequence in its bounded derived category. In fact, Krause \cite[Thm 5.2]{MR3742477} shows they are {\em strictly} full, \ie their extension closure in $\cat{A}$ is an exact category derived equivalent to $\cat{A}$. This exact category was studied earlier by Dlab and Ringel \cite{MR1211481}. More recently, Bodzenta and Bondal \cite[Thm 5.8]{MR4780036} characterised the exact categories which arise in this way, and described how highest-weight categories can be viewed as their abelian envelopes.

In the next section, we briefly review the definition and properties of highest-weight categories, and their characterisations via recollements, exceptional sequences, exact categories and quasi-hereditary algebras. In section \ref{sec: stability functions} we recall the notion of a stability function on an abelian category.

We prove our main result, Theorem \ref{thm:new hwc characterisations}, in \S\ref{sec: main thm}. For expository purposes we split it into two results here. Let $\cat{A}$ be a Deligne-finite $k$-linear category over an algebraically-closed field $k$. Fix an ordering  $S_1,\ldots,S_n$ of its simple objects. Let $P_i$ be corresponding projective covers, and $\stan{i}$,  respectively $\costan{i}$, the corresponding standard,  respectively  costandard, objects. Finally, let $[A]$ denote the class of $A\in \cat{A}$ in the Grothendieck group, and $[A:S]$ the  multiplicity of the simple object $S$ in a Jordan--H\"older filtration of $A$. Our first characterisation of highest weight is the following numerical criterion:

\newtheorem*{thm1}{Theorem 1}
\begin{thm1}
\label{thm:1}
    The category $\cat{A}$ is highest weight if, and only if, $[\stan{i}:S_i]=1$ and $[P_i]=\sum_{j=1}^n[\costan{j}:S_i]\, [\stan{j}]$ for each $i=1,\ldots,n$.
\end{thm1}

The necessity of this criterion follows directly from the definition of highest-weight category and Bernstein--Gelfand--Gelfand reciprocity. To prove its sufficiency, we consider the inductive construction of the projective covers over the recollements. The numerical condition on the class of the projective cover implies it has the maximum possible length, and the construction then shows this occurs when it has a filtration by standard objects.  

The second part of our main result is the following characterisation of highest weight in terms of stability functions:
\newtheorem*{thm2}{Theorem 2}
\begin{thm2}
\label{thm:2}
    The category $\cat{A}$ is highest weight if, and only if, it admits a stability function in which the $\stan{i}$ are stable of strictly increasing phase and the Harder--Narasimhan factors of each $P_i$ are direct sums of $\stan{j}$ for $j\geq i$.
\end{thm2}

The sufficiency of this condition is immediate; its necessity follows from Theorem \ref{thm:1}, which ensures we can construct a stability function in which the standard objects are stable. This is not surprising; the standard objects in a highest-weight category satisfy $\End{\cat{A}}{\stan{i}}\cong k$ and $\Mor{\cat{A}}{\stan{i}}{\stan{j}} = 0$ for $i>j$, so behave just like a sequence of stable objects of strictly increasing phase. Moreover, each $P_i$ has a `good' filtration by standard objects occurring in strictly decreasing order, which behaves just like a Harder--Narasimhan filtration. Theorem \ref{thm:2} confirms this behaviour is not accidental. 

In \S\ref{sec:properly-stratified} we generalise the stability criterion to a characterisation of {\em properly stratified} categories \cite{MR1781825}. This is a weaker notion in which we drop the requirement of $k$-linearity and only require that each projective has a filtration by {\em proper} standard objects --- see Definition~\ref{def:psc}. 

In \S\ref{sec:green-schroll} we recall Green and Schroll's  necessary and sufficient condition for the category of modules over a monomial algebra $\Lambda$ to be highest weight \cite{MR4017924}. Their condition is beautifully simple to check if one has a quiver description as it only involves checking the relations, and does not require the computation of projective and standard objects. In Corollary \ref{cor: gs} we give an alternative proof based on the numerical criterion from Theorem \ref{thm:1}.  More generally, any finite-dimensional algebra $\Lambda$ has an associated monomial algebra $\Lambda_\text{mon}$, isomorphic to $\Lambda$ as a vector space but with a modified multiplication. Green and Schroll show that $\mod{\Lambda}$ is highest weight whenever $\mod{\Lambda_\text{mon}}$ is. However, the converse fails, see Example \ref{ex:4}.

Both criteria use the fact that we have finitely many simple objects so that projective classes are finite sums of standard classes in the Grothendieck group, and the corresponding Harder--Narasimhan filtrations are finite. 

Finally, in \S\ref{sec: examples}, we present some simple examples which illustrate our results. We present them via quivers with relations, but the first four can also be interpreted as a category of perverse sheaves. Indeed, this paper is part of an ongoing project to understand the topological conditions under which a category of perverse sheaves (for an arbitrary perversity) is highest weight.

\subsection*{Acknowledgements}
The first author was funded by MUR PNRR--Seal of Excellence, CUP B37G22000800006 and supported by the ``National Group for Algebraic and Geometric Structures, and their Applications" (GNSAGA - INdAM). The authors also acknowledge the project funded by NextGenerationEU under NRRP, Call PRIN 2022 No. 104 of February 2, 2022 of Italian Ministry of University and Research; Project 2022S97PMY Structures for Quivers, Algebras and Representations (SQUARE). 

We thank Alexey Bondal for pointing out that (\ref{eqn:standard-costandard exts}) holds for highest weight categories, but not in the greater generality we had previously claimed. We also thank the anonymous referee for a number of helpful suggestions, inlcuding that we consider the case of properly stratified categories.

\section{Highest-weight categories}
\label{sec: hwc}

We recall the definition of highest-weight category, and summarise characterisations for a category to be highest weight due to Krause \cite{MR3742477} and Bodzenta--Bondal \cite{MR4780036}. We work with coefficients in an algebraically-closed field $k$. It is possible to define highest-weight categories more generally, as for example Krause does in \cite{MR3742477}, but for simplicity, and because it covers the examples we are primarily interested in, we do not do so. 

Let $\cat{A}$ be a $k$-linear abelian category. Further we assume that $\cat{A}$ is \defn{Deligne finite} as defined in \cite[Appendix A.3]{MR4780036}, \ie  that $\cat{A}$
\begin{enumerate}
    \item is $\mathrm{Hom}$ and $\mathrm{Ext}^1$-finite,
    \item is a length category, \ie Noetherian and Artinian,
    \item has finitely many simple objects and
    \item has enough projective objects.
\end{enumerate} 
The examples we have in mind are the category of perverse sheaves (for any perversity) on a stratified space with finitely many strata, each with finite fundamental group, and the category of finite-dimensional modules over a finite-dimensional algebra. Indeed, any Deligne-finite category $\cat{A}$ is equivalent to the category of finite-dimensional modules over $\End{\cat{A}}{P}$ where $P$ is a projective generator. In particular, $\cat{A}$ also has enough injective objects because $\End{\cat{A}}{P}$ is a finite-dimensional algebra.

Choose a total ordering $S_1,\ldots,S_n$ of (representatives of the iso-classes of) the simple objects of $\cat{A}$. This determines a filtration
\[
0=\cat{A}_0 \subset \cat{A}_1 \subset \cat{A}_2 \subset \cdots \subset \cat{A}_n=\cat{A}
\]
by the Serre subcategories $\cat{A}_i = \extcl{S_1,\ldots,S_i}$ given by the extension-closure of the first $i$ simple objects. For each $i$, let $P_i$ and $I_i$ be respectively a projective cover and an injective hull of $S_i$ in $\cat{A}$.
\begin{definition}
The \defn{standard object} $\stan{i}$ is the maximal quotient of $P_i$ lying in the subcategory $\cat{A}_i$, and the \defn{costandard object} $\costan{i}$ is the maximal subobject of $I_i$ lying in the subcategory $\cat{A}_i$. 
\end{definition}
The standard and costandard objects exist because $\cat{A}$ is length, and are non-zero because $S_i$ is a quotient of $P_i$ and a subobject of $I_i$. They depend upon the choice of ordering of the simple objects. This dependence is also apparent in the following alternative characterisation.
\begin{lemma}
\label{lem:standards are projectives}
    The standard object $\stan{i}$ is a projective cover, and the costandard object $\costan{i}$ an injective hull, of $S_i$ in $\cat{A}_i$.
\end{lemma}
\begin{proof}
    The two statements are dual; we prove only the first. Since $\stan{i}$ is a quotient of $P_i$, the group $\Mor{\cat{A}_i}{\stan{i}}{S_j}\cong k$ for $j= i$ and vanishes for $0\leq j<i$. Therefore it suffices to show that $\Ext{\cat{A}_i}{1}{\stan{i}}{S_j}=0$ for $0\leq j \leq i$. Consider a short exact sequence
    \[
0 \to S_j \to E \to \stan{i} \to 0
    \]
    in $\cat{A}_i$, thus also in $\cat{A}$. Applying $\Mor{\cat{A}}{P_i}{-}$ shows that the quotient morphism $P_i \to \stan{i}$ factors through $E$. Since $\stan{i}$ is the maximal quotient in $\cat{A}_i$ the short exact sequence splits. Hence $\Ext{\cat{A}_i}{1}{\stan{i}}{S_j}=0$ for $0\leq j \leq i$.
\end{proof}

\begin{definition}
\label{def:hwc}
    The category $\cat{A}$ is \defn{highest weight with respect to the chosen ordering of its simple objects} if $\End{\cat{A}}{\stan{i}}\cong k$ and $P_i \in \extcl{\stan{i},\ldots,\stan{n}}$ for each $i=1,\ldots,n$. We say $\cat{A}$ is \defn{highest weight} if it is so for some ordering.
\end{definition}
\begin{remarks} This is the definition used in \cite[\S5.2]{MR4780036} except that we assume the simple objects are totally, not just partially, ordered. This is not as restrictive as it first appears. If $\cat{A}$ is highest weight with respect to a poset $\Lambda$ then $\Lambda$ is \defn{adapted} to $\cat{A}$ by \cite[Prop 1.4.12]{CondeThesis}. This means that if $A\in \cat{A}$ has top $S_\lambda$ and socle $S_\mu$, both simple and with $\lambda,\mu\in \Lambda$ incomparable, then $A$ has a composition factor $S_\nu$ with $\nu >\lambda$ or $\nu>\mu$. The importance of this condition is that it guarantees the stability of the standard and costandard objects under refinements of $\Lambda$. Therefore, if $\cat{A}$ is highest weight with respect to $\Lambda$, it is also highest weight for all total orders refining $\Lambda$. Thus there is no loss of generality in considering only total orders, but there may be a loss of naturality. For example, the simple perverse sheaves on a stratified space are naturally partially ordered by their supports.

    Our definition privileges the role of the standard objects, but in fact $\cat{A}$ is highest weight if and only if the opposite category is highest weight (for the same ordering of simple objects). This leads to an equivalent dual definition in terms of the costandard and injective objects, see for example \cite{Cline1988}. In the sequel we will always give the  characterisation from each dual pair involving the standard objects.
\end{remarks}

In Theorem \ref{thm:old hwc characterisations} we will give four alternative characterisations of highest-weight categories. In order to do so we introduce some terminology. 

\subsection*{Homological recollements}
An \defn{abelian recollement} \cite[1.4]{BBD} is a diagram
\[
\begin{tikzcd}
\cat{B} \ar{rr}{\imath_*} &&  
\cat{C} \ar{rr}{\jmath^*} \ar[bend right]{ll}[swap]{\imath^*} \ar[bend left]{ll}[swap]{\imath^!}
&& \cat{D} \ar[bend right]{ll}[swap]{\jmath_!} \ar[bend left]{ll}[swap]{\jmath_*}
\end{tikzcd}
\]
of abelian categories in which 
\begin{enumerate}
    \item $\imath_*$ is exact with respective left and right adjoints $\imath^*$ and $\imath^!$,
    \item $\jmath^*$ is exact with respective left and right adjoints $\jmath_!$ and $\jmath_*$,
    \item  $\imath_*$ is fully faithful with essential image the kernel of $\jmath^*$ and
    \item the counit and unit $\jmath^*\jmath_! \to \id_\cat{C}\to \jmath^*\jmath_*$ are isomorphisms.
\end{enumerate}
An abelian recollement is \defn{homological} \cite{MR3123754} if in addition the fully-faithful embedding $\imath_*$ induces isomorphisms
\[
\Ext{\cat{B}}{r}{B}{B'}\cong \Ext{\cat{C}}{r}{\imath_*B}{\imath_*B'}
\]
for all $B,B'\in \cat{B}$ and $r\geq0$. 

\subsection*{Strictly full exceptional sequences}
An object $E$ of a $k$-linear abelian category $\cat{A}$ is \defn{exceptional} if $\End{\cat{A}}{E}\cong k$ and $\Ext{\cat{A}}{k}{E}{E}=0$ for $k>0$. A sequence of objects $(E_1, \ldots , E_n)$ in $\cat{A}$ is \defn{exceptional} if each $E_i$ is exceptional and $\Ext{\cat{A}}{r}{E_i}{E_j} = 0$ for $i>j$ and $r \geq 0$. The sequence is \defn{full} if the objects $E_1, \ldots , E_n$ generate the bounded derived category $\der{b}{\cat{A}}$ as a triangulated category, and \defn{strictly full} if the inclusion $\extcl{E_1,\ldots,E_n} \monic \cat{A}$ induces a triangle equivalence 
\[
\der{b}{\extcl{E_1,\ldots,E_n} } \to \der{b}{\cat{A}}
\]
from the bounded derived category of the exact subcategory $\extcl{E_1,\ldots,E_n}$ to that of $\cat{A}$. As the terminology suggests, the latter is a strictly stronger condition; see \cite[Examples 5.9 and 5.10]{MR3742477} for full exceptional sequences which are not strictly full. 

\subsection*{Thin exact categories}
A $k$-linear, $\mathrm{Hom}$ and $\mathrm{Ext}^1$-finite exact category $\cat{E}$ is \defn{thin} \cite[\S3.1]{MR4780036} if it has a right admissible filtration 
\[
0 = \cat{E}_0 \subset \cat{E}_1 \subset \cdots \subset \cat{E}_n = \cat{E}
\]
whose graded factors are equivalent to $\kvect$ for $i=1,\ldots,n$. A full exact subcategory $\cat{T}\subset \cat{E}$ is \defn{right admissible} if $(\cat{T},\cat{T}^\perp)$ is a torsion pair in $\cat{E}$ where
\[
\cat{T}^\perp = \{ E\in \cat{E} \mid \Mor{\cat{E}}{T}{E}=0=\Ext{\cat{E}}{1}{T}{E}\ \forall\, T\in \cat{T} \}.
\]
Roughly, this is the exact category version of a homological recollement; see \cite{MR4780036} for full details.

\subsection*{Quasi-hereditary algebras}

A finite-dimensional $k$-algebra $\Lambda$ is \defn{quasi-hereditary} if there is a chain of two-sided ideals
\[
0 = J_0 \subset J_1 \subset \cdots \subset J_{n-1}\subset J_n = \Lambda
\]
such that $J_i/J_{i-1}$ is a heredity ideal in $\Lambda/J_{i-1}$ for all $i$. Here  $J$ is a heredity ideal in $\Lambda$  if 
\begin{enumerate}[(i)]
    \item $J^2=J$, 
    \item $JRJ=0$ where $R$ is the Jacobson radical of $\Lambda$ and 
    \item $J$ is projective as a left or right $\Lambda$-module.
\end{enumerate}

\begin{theorem}
\label{thm:old hwc characterisations}
The following are equivalent:
\begin{enumerate}
    \item The category $\cat{A}$ is highest weight with respect to the chosen ordering of simple objects.
    \item The diagram
$\cat{A}_{i-1} \monic \cat{A}_i \epic \cat{A}_i/\cat{A}_{i-1}$
    determines a homological recollement with $\cat{A}_i/\cat{A}_{i-1}\simeq \kvect$ for each $i=1,\ldots,n$. 
    \item The standard objects $(\stan{1},\ldots,\stan{n})$ form a strictly-full exceptional sequence in $D^b(\cat{A})$. 
    \item The extension-closure $\extcl{\stan{1},\ldots,\stan{n}}$ is a thin exact category for the filtration defined by the chosen ordering.
\end{enumerate}
Moreover, $\End{\cat{A}}{\bigoplus_{i=1}^n P_i}$ is a quasi-hereditary algebra if, and only if, $\cat{A}$ is highest weight with respect to some ordering of its simple objects.
\end{theorem}
\begin{proof}
The equivalence of (1) and (2) is Krause \cite[Thm 3.4]{MR3742477}. 

The equivalence of (1) and (3) is also essentially due to Krause. In one direction, \cite[Prop 5.7]{MR3742477} states that the standard objects of a highest-weight category form a strictly full exceptional sequence in $D^b(\cat{A})$. In the other direction, \cite[Thm 5.2]{MR3742477} states that if $E_1,\ldots,E_n$ is a strictly full exceptional sequence in $D^b(\cat{A})$ then there is a bounded heart $\cat{A}'$ in $D^b(\cat{A})$ such that $D^b(\cat{A}') \simeq D^b(\cat{A})$ and which is highest weight with $E_1,\ldots,E_n$ as its standard objects. Applying this in the case when $E_i=\stan{i}$ for $i=1,\ldots,n$ yields a highest-weight heart $\cat{A}'$ in which the $\stan{i}$ are standard objects. We claim that $\cat{A}=\cat{A}'$. To see this note that $\cat{A}'$ has a projective generator $P'$ in terms of which 
\[
\cat{A}' = \{ A\in D^b(\cat{A}) \mid \Ext{}{\neq 0}{P'}{A}=0\}.
\]
Since $\cat{A}'$ is highest weight,  $P' \in \Filt{\stan{1},\ldots,\stan{n}} \subset \cat{A}$ so that
\[
\Ext{}{<0}{P'}{A}=0 \quad \forall A\in \cat{A}.
\]
To complete the proof we show $\Ext{}{>0}{P'}{S_i}=0$ for $i=1,\ldots,n$ by induction. Note that $\End{\cat{A}}{\stan{i}}\cong k$ since $\stan{i}$ is exceptional, so there are short exact sequences
\[
0 \to K_i \to \stan{i} \to S_i \to 0
\]
in $\cat{A}$ with kernel $K_i\in \cat{A}_{i-1}$. The base case follows since $S_1=\stan{1} \in \cat{A}'$, and the inductive step follows because $S_i \in \Filt{\stan{i}} * \cat{A}_{i-1}[1]$. Hence $\cat{A} \subset \cat{A}'$, and so  $\cat{A}=\cat{A}'$,  because nested bounded hearts are equal. 

The equivalence of (1) and (4) is proved in Bodzenta--Bondal  \cite[\S 5]{MR4780036}.

Finally, the relationship between quasi-hereditary algebras and highest-weight categories is due to Cline, Parshall and Scott \cite[Thm 3.6]{Cline1988}.
\end{proof}
\begin{remark}
    The second condition can be sharpened: Wiggins \cite[Corollary 6.6]{wiggins2025stratificationsabeliancategories} shows that $\cat{A}$ is highest weight if and only if each $\cat{A}_{i-1} \monic \cat{A}_i \epic \cat{A}_i/\cat{A}_{i-1}$ is a $2$-homological recollement, \ie there are isomorphism of Ext groups as in the definition of homological recollement only for $r\leq 2$. 
\end{remark}

We end the section by recalling some well-known properties of  highest-weight categories.
\begin{theorem}[{\cite[Thm 4.3]{PS}, \cite[Thms 1, 2]{MR0987824}, \cite[Ex 3.3]{Cline1988}}]
    If $\cat{A}$ is highest weight then $\gldim{\cat{A}}<\infty$. Moreover, $\cat{A}$ is highest weight for all orderings of its simple objects if and only if $\gldim{\cat{A}}\leq 1$, and if $\gldim{\cat{A}}\leq 2$ then it is highest weight for some ordering.
\end{theorem}
The Grothendieck group $K(\cat{A})$ of a Deligne-finite category $\cat{A}$ is free of finite rank. When $\gldim{\cat{A}}<\infty$ there are three natural integral bases given by the classes of the simple objects, of their projective covers and of their injective hulls. Moreover, the Euler form
\[
\chi(A,B) = \sum_{r\in\Z} \Ext{\cat{A}}{r}{A}{B}
\]
is a well defined, non-degenerate bilinear form on $K(\cat{A})$. The above bases are dual with respect to the Euler form: 
\[
\chi([P_i],[S_j])=\delta_{ij}=\chi([S_i],[I_j]).
\]
When $\cat{A}$ is highest weight the classes of the standard and costandard objects give two further integral bases of $K(\cat{A})$. Moreover, in this case
\begin{equation}
    \label{eqn:standard-costandard exts}
\Ext{\cat{A}}{r}{\stan{i}}{\costan{j}} \cong
\begin{cases}
    k & i=j, r=0\\
    0 & \text{otherwise}
\end{cases}
\end{equation}
so that $\chi([\stan{i}],[\costan{j}])=\delta_{ij}$, \ie these bases are also dual.

Since $\cat{A}$ is length each object admits a Jordan--H\"older filtration with simple factors. We denote the multiplicity of the simple $S_i$ in the filtration of $A$ by $[A:S_i]$. By definition, when $\cat{A}$ is highest weight, the projective $P_i$ also admits a filtration whose factors are standard objects. Let $(P_i:\stan{j})$ denote the multiplicity of $\stan{j}$ as a factor in this filtration. These multiplicities can be interpreted as coefficients in appropriate change of basis matrices. A simple calculation using (\ref{eqn:standard-costandard exts}) shows that
\begin{equation}
\label{BGG reciprocity}
    (P_i:\stan{j}) = \chi([P_i],[\costan{j}]) = [\costan{j}:S_i].
\end{equation}
This is known as Bernstein--Gelfand--Gelfand reciprocity because the corresponding fact in the case of semisimple Lie algebras was first proved in \cite{MR0407097}. See also \cite[Thm 3.11]{Cline1988} where the dual statement, which they refer to as Brauer--Humphreys reciprocity, is proved in the abstract setting. 

\section{Stability functions}
\label{sec: stability functions}

We recall the notion of a stability function on an abelian category. See \cite{MR2373143} for further details, and for the relation between stability conditions on a triangulated category $\cat{D}$ and stability functions on the hearts of bounded t-structures on $\cat{D}$.

Let $\cat{A}$ be an abelian category and let $\HH=\{ r e^{i\pi\phi} \mid r>0, \phi\in(0,1] \}$ denote the upper half plane with the positive $x$-axis removed.  A \defn{stability function} on $\cat{A}$ is an additive group homomorphism $Z\colon K(\cat{A})\to\C$ such that $Z(A):=Z([A])\in\HH$ for any $A\in\cat{A}$. The \defn{phase} $0<\phi(A)\leq 1$ of an object $A$ is uniquely defined by the requirement $Z(A) \in \R_{>0}e^{i\pi\phi(A)}$.

An object $A\in\cat{A}$ is \defn{$Z$-semistable} if  $\phi(A')\leq\phi(A)$ for any proper subobject $A' \subset A$, and \defn{$Z$-stable} if the inequality is always strict. The full subcategory of semistable objects of phase $\phi$ is denoted $\cat{P}(\phi)$. It is a wide subcategory of $\cat{A}$, \ie it is closed under extensions, kernels and cokernels, and in particular is a full abelian subcategory. The following fact is well-known.

\begin{lemma}
\label{lem: end of stable}
If $S$ is stable then $\End{\cat{A}}{S}\cong k$.
\end{lemma}
\begin{proof}
    Any endomorphism of a stable object $S$ must be an isomorphism because if the image were a proper subobject it would have phase both strictly greater and strictly lesser than that of $S$, which is impossible. Thus $\End{\cat{A}}{S}$ is a division algebra over $k$. Hence $\End{\cat{A}}{S}\cong k$ because $k$ is algebraically closed.
\end{proof}

We now discuss two important properties of stability functions, namely the Harder--Narasimhan and support properties. Both are automatic in the context we work in of Deligne finite categories. 

A stability function has the \defn{Harder-Narasimhan property} if each non zero object $A\in\cat{A}$ has a filtration
\[
0=A_0\monic A_1 \monic A_2\monic \ldots \monic A_d=H
\]
such that $A_i/A_{i-1}\in\cat{P}(\phi_i)$ with $\phi_1>\phi_2>\ldots>\phi_d$. Such a filtration is unique if it exists; its factors are referred to as the \defn{Harder--Narasimhan factors} of $A$. When $\cat{A}$ is a length abelian category any stability function has the Harder--Narasimhan property by \cite[Proposition 2.5]{MR2373143}. 

Suppose that $Z$ is a stability function on $\cat{A}$ which factorises through a finite rank free quotient $\lambda \colon K(\cat{A}) \to \Lambda$. Then $Z$ has the \defn{support property} if there is a constant $C>0$ such that 
\[
\lvert Z(A) \rvert\geq C\lVert \lambda([A]) \rVert
\]
for any semistable object $A\in\cat{A}$. Here $\lVert - \rVert$ is a choice of norm on $\Lambda\otimes\R$; the support property is independent of the choice. When $Z$ satisfies the support property the subcategories $P(\phi)$ of semistable objects are length abelian subcategories of $\cat{A}$. Each semistable object has a Jordan--H\"older filtration whose factors are stable objects of the same phase.

When $\cat{A}$ is a length abelian category with finitely many iso-classes of simple objects $K(\cat{A})$ is a finite rank free abelian group, with a basis given by the classes of the simple objects. In this case, we may choose $\Lambda=K(\cat{A})$ and any stability function has the support property because we may choose $C$ to be the minimum of $|Z(S)| / \lVert [S] \rVert$ over the simple objects $S$ of $\cat{A}$.

\section{Two new characterisations of highest-weight categories}
\label{sec: main thm}
We give two new characterisations of highest-weight categories, the first numerical in terms of the Grothendieck group and the second in terms of the existence of a special kind of Bridgeland stability condition.

\begin{theorem}
\label{thm:new hwc characterisations}
    Let $\cat{A}$ be a Deligne finite $k$-linear abelian category. Fix an ordering $S_1,\ldots,S_n$ of its simple objects. Then the following are equivalent:
    \begin{enumerate}[(i)]
        \item $\cat{A}$ is highest weight with respect to the given ordering;
        \item $[\stan{i}:S_i]=1$ and $[P_i] = \sum_{j=i}^n [\costan{j}:S_i] [\stan{j}]$ for $i=1,\ldots,n$;
        \item there is a stability function $Z \colon K(\cat{A}) \to \C$ on $\cat{A}$ such that
        \begin{enumerate}
            \item $\stan{1}, \ldots,\stan{n}$ stable with $\phi(\stan{1}) < \cdots < \phi(\stan{n})$ and
            \item the Harder--Narasimhan factors of each $P_i$ are direct sums of the standard objects $\stan{j}$ for $i\leq j$.
        \end{enumerate}
    \end{enumerate}
\end{theorem}

\begin{proof}
    (i)$\implies$(ii): By definition $\End{\cat{A}}{\stan{i}} \cong k$ and $P_i \in \Filt{\stan{i},\ldots,\stan{n}}$. By Lemma \ref{lem:standards are projectives} the standard object  $\stan{i}$ is a projective cover of $S_i$ in $\cat{A}_i$. Therefore $[\stan{i}:S_i] =  \dim \End{\cat{A}}{\stan{i}} = 1$. Moreover,  
    \[
[P_i] = \sum_{j=i}^n (P_i:\stan{j})[\stan{j}] = \sum_{j=i}^n [\costan{j}:S_i][\stan{j}]
    \]
     by Bernstein--Gelfand--Gelfand reciprocity (\ref{BGG reciprocity}).
     
     (ii)$\implies$(iii): Choose phases $0=\phi_0 < \phi_1 < \cdots < \phi_n <1$. The classes of the simple objects form a basis for $K(\cat{A})$ so a stability function $Z$ is uniquely determined by setting $Z(S_i)=m(S_i)e^{i\pi\phi_i}$ for $i=1,\ldots,n$ where $m(S_i)>0$. Since $\stan{i} \in \cat{A}_i$ with $[\stan{i}:S_i] =1$ one can choose the $m(S_i)$ so that 
     \[
\phi_{i-1} < \phi(\stan{i}) \leq \phi_i
     \]
     for each $i=1,\ldots,n$, with $\phi(\stan{i})=\phi_i$ if and only if $\stan{i}=S_i$. Since every proper subobject of $\stan{i}$ lies in $\cat{A}_{i-1}$ it follows that each $\stan{i}$ is stable.

     Let $\imath_* \colon \cat{A}_{n-1} \hookrightarrow \cat{A}$ be the inclusion and $\jmath^* \colon \cat{A} \to \cat{A}/\cat{A}_{n-1} \simeq \kvect$ the quotient. For each projective cover $P_i$ there is an exact sequence
     \begin{equation}
         \label{eqn: es from recollement}
\jmath_!\jmath^*P_i \to P_i \to \imath_*\imath^*P_i \to 0.
      \end{equation}
     Since $\jmath_!$ is left adjoint to the exact functor $\jmath^*$ it preserves projective objects. Therefore $j_!j^*P_i \cong \stan{n}^{[\costan{n}:S_i]}$ because $\jmath_!\jmath^*S_n \cong P_n = \stan{n}$ and
         \[
         [P_i:S_n] = \dim \Mor{}{P_i}{\costan{n}} = [\costan{n}:S_i]
         \]
         using the dual fact that $I_n=\costan{n}$. Combining this with (\ref{eqn: es from recollement}) shows that the lengths, \ie the numbers of simple objects in the Jordan--H\"older filtration, satisfy
    \[
    \length{P_i} \leq [\costan{n}:S_i]\cdot \length{\stan{n}} + \length{\imath^*P_i}
    \]
    with equality if and only if (\ref{eqn: es from recollement}) is short exact.

 Now using the fact that $\imath^*$ is left adjoint to the exact functor $\imath_*$, and so preserves projective objects, we conclude that $\imath^*P_i$ is the projective cover of $S_i$ in $\cat{A}_{n-1}$ when $i<n$ and vanishes when $i=n$. Thus we can apply the above argument inductively to show that
 \begin{equation}
 \label{eqn: length inequality}
     \length{P_i} \leq \sum_{j=1}^n [\costan{j}:S_i]\cdot \length{\stan{j}}.
 \end{equation}
The assumption $[P_i] = \sum_{j=i}^n [\costan{j}:S_i][\stan{j}]$
 implies this is an equality. Therefore each of the exact sequences used to derive (\ref{eqn: length inequality}) is actually short exact. These short exact sequences provide the desired Harder--Narasimhan filtration of $P_i$.

     (iii)$\implies$(i): Since $\stan{i}$ is stable $\End{\cat{A}}{\stan{i}} \cong k$ by Lemma \ref{lem: end of stable}. The condition $P_i \in \Filt{\stan{i},\ldots,\stan{n}}$ follows immediately from the description of the Harder--Narasimhan filtration of $P_i$.
\end{proof}
See Section~\ref{sec: examples} for examples. 
\begin{remarks}
\begin{enumerate}
    \item The proof shows that $\cat{A}$ is highest weight when the indecomposable projective objects have the maximal possible length. 
    \item If the Harder--Narasimhan factors of a projective object are direct sums of standard objects for some stability function in which the standard objects are stable with $\phi(\stan{1})< \cdots < \phi(\stan{n})$, then they are so for all such stability functions, because the Harder--Narasimhan filtration is unique. Therefore it suffices to check the third condition for any such stability function.
\end{enumerate}
\end{remarks}

\section{Properly stratified categories}
\label{sec:properly-stratified}

The characterisation of highest weight categories in terms of stability conditions in Theorem~\ref{thm:new hwc characterisations}  generalises to properly stratified abelian categories, but we do not know of a numerical criterion in this generality.

Let $\cat{A}$ be a length abelian category, not necessarily $k$-linear, with finitely many iso-classes of simple objects and enough projectives. Choose a total ordering $S_1,\ldots,S_n$ of  representatives of the iso-classes of simple objects. Define  the standard object $\stan{i}$ as before to be the maximal length quotient of the projective cover $P_i$ of $S_i$ in $\Filt{S_j \mid j\leq i}$. Define the \defn{proper standard} object $\pstan{i}$ to be the maximal length quotient of $\stan{i}$ with $[\pstan{i} \colon S_i]=1$. The proper standard objects are well-defined up to isomorphism. Note that $\Mor{}{\pstan{j}}{\pstan{i}}=0$ when $i<j$ because $\Mor{}{P_j}{\pstan{i}}=0$. The following definition is a genralisation of the notion of (the module category of) a \defn{properly stratified algebra} \cite[Defn.~4]{MR1781825}. When the properly stratified category is also Hom-finite then it is equivalent to the module category of a properly stratified algebra, namely the endomorphism algebra of a projective generator. 
\begin{definition}
\label{def:psc}
The category $\cat{A}$ is \defn{properly stratified} with respect to the chosen order if $P_i \in \Filt{ \pstan{j} \mid i \leq j}$  for each $i = 1,\ldots,n$.
\end{definition}
This  is a generalisation of Definition~\ref{def:hwc}, and the following is the accompanying generalisation of the stability function characterisation of highest weight categories in Theorem~\ref{thm:new hwc characterisations}. 
See Example~\ref{ex: 5} for a properly stratified category which is not highest weight. 
\begin{theorem}
\label{thm:properly stratified}
A length abelian category $\cat{A}$ with finitely many iso-classes of simples and enough projectives is standardly stratified if and only if there exists a stability function $Z \colon K(\cat{A}) \to \C$ such that
\begin{enumerate}
\item the proper standard objects $\pstan{i}$ are stable with  $\phi(\pstan{i}) <\phi(\pstan{j})$ when $i < j$ and
\item the Harder--Narasimhan factors of each $P_i$ are self-extensions of $\pstan{j}$ for $i\leq j$.
\end{enumerate}
\end{theorem}
\begin{proof}
Suppose $\cat{A}$ is properly stratified. Using the fact that $[\pstan{i} \colon S_i]=1$ we can construct a stability function $Z \colon K(\cat{A}) \to \C$ in which the proper standards are stable of strictly increasing phase just as in the proof of Theorem~\ref{thm:new hwc characterisations}. By assumption, each indecomposable projective $P_i$ has a filtration by proper standard objects $\pstan{j}$ where $i\leq j$. Since $\Mor{}{\pstan{j}}{\pstan{k}}=0$ when $j>k$ this filtration can be ordered by decreasing index $j \in\{i,\ldots,n\}$. It is then a filtration with semistable factors, each a self-extension of a proper standard $\pstan{j}$ with $i\leq j$,  of strictly decreasing phase. By uniqueness it is the Harder--Narasimhan filtration, which therefore has the required form.

Conversely, suppose there is such a stability function $Z$. Then the Harder--Narasimhan filtrations of the indecomposable projectives are the required  filtrations by proper standard objects. Hence $\cat{A}$ is properly stratified.
\end{proof}
\begin{remark}
Since the proper standard objects are stable they are bricks. This is immediate from their definition when the simple objects are totally ordered, but may fail to be the case if they are only partially ordered. We do not know if categories which are properly stratified with respect to a partial order can be characterised by the existence of stability functions as above for which the proper standard objects are semistable, but not necessarily stable.
\end{remark}

\section{Green--Schroll's sufficient criterion}
\label{sec:green-schroll}

Green and Schroll \cite{MR4017924} give a necessary and sufficient criterion for a finite-dimensional monomial algebra to be quasi-hereditary, equivalently for its module category to be highest weight. They also show this leads to a sufficient condition for a finite-dimensional algebra to be quasi-hereditary by applying it to an associated monomial algebra. Their proof uses non-commutative Gr\"obner basis theory; we give an alternative derivation of their necessary and sufficient condition from our numerical characterisation.

Let $Q$ be a finite quiver without loops and with vertices $1,\ldots,n$. Let $I \subset kQ$ be an admissible ideal which is generated by paths, rather than by sums of paths. Then $kQ/I$ is a finite-dimensional monomial algebra. Let $\cat{A} = \mod{kQ/I}$ and use the previous notation for its simple, projective, injective, standard and costandard objects, the latter with respect to the ordering of simple objects by the corresponding vertices. 

Let $s,t \colon Q_1 \to Q_0$ be respectively the source and target maps from arrows to vertices of $Q$. A \defn{path} of length $d$ in $Q$ is a concatenation $\alpha_1\ldots \alpha_d$ of arrows where $t(\alpha_i) = s(\alpha_{i+1})$ for $i=1,\ldots, d-1$; a \defn{cycle} is a path for which $s(\alpha_1)=t(\alpha_d)$. The admissible ideal $I$ has a canonical generating set $G$ consisting of the minimal length paths in $I$. 

We refer to $t(\alpha_1),\ldots,t(\alpha_{d-1})$ as \defn{internal}, and $s(\alpha_1)$ and $t(\alpha_d)$ as \defn{external}, vertices of the path $\alpha_1\ldots \alpha_d$. We say a vertex $i$ in a path is \defn{maximal} if $i \geq s(\alpha), t(\alpha)$ for all arrows $\alpha$ in the path. Let $Q_{\leq k}$ denote the full subquiver on vertices $1,\ldots,k$, and $\pi^k_{ij}$ be the number of paths from $i$ to $j$ in $Q_{\leq k}$ which are not in the ideal $I$. 

The third condition in the following result is Green and Schroll's criterion.
\begin{corollary}
\label{cor: gs}
The following are equivalent:
\begin{enumerate}[(i)]
    \item $\mod{kQ/I}$ is highest weight for the given ordering;
    \item $\pi_{ii}^i=1$ and $\pi_{ik}^n = \sum_{j=1}^n \pi_{ij}^j\pi_{jk}^j$ for each $1\leq i,k\leq n$;
    \item all maximal vertices of each path in $G$ are external.
    \end{enumerate}
\end{corollary}
\begin{proof}
    (i)$\iff$(ii): Since $I$ is generated by paths one can reinterpret certain counts of paths as Jordan--H\"older multiplicities. In particular $\pi_{ik}^n = [P_i:S_k]$, and similarly
    \[
\pi_{ij}^j = [S_i : \costan{j}] \quad \text{and} \quad \pi_{jk}^j = [\stan{j}:S_k]
\]
because $\costan{j}$ and $\stan{j}$ are respectively the injective hull and projective cover of $S_j$ in $\Filt{S_1,\ldots,S_j}$ by Lemma \ref{lem:standards are projectives}. The equivalence then follows from Theorem \ref{thm:new hwc characterisations} and the observation that
\[
[P_i] = \sum_{j=1}^n [\costan{j}:S_i]\, [\stan{j}]
\]
if and only if $[P_i:S_k] = \sum_{j=1}^n [\costan{j}:S_i]\, [\stan{j}:S_k]$ for each $k=1,\ldots,n$.

    (ii)$\iff$(iii): Suppose all maximal vertices of each path in the generating set $G$ are external. Then $\pi_{ii}^i=1$ because if $\gamma \not \in I$ is a strictly positive length cycle at $i$ in $Q_{\leq i}$, then the minimal power of $\gamma$ which is in $I$ would contain a path in $G$ with $i$ as a maximal internal vertex. 
    
    It follows that each path not in $I$ has a unique maximal vertex: for if the maximal vertex, $i$ say, occurs twice then the path contains a strictly positive length cycle at $i$ in $Q_{\leq i}$ which is not in $I$. Moreover, the concatenation of two  paths not in $I$ at their maximal vertices, $i$ say,  cannot be in $I$ because if it were it would contain a path in $G$ with $i$ as an maximal internal vertex. Combining these two facts we see that
    \[
\pi_{ik}^n  = \sum_{j=1}^n \# \{ \text{paths $i$ to $k$ not in $I$ with maximal vertex $j$}\} = \sum_{j=1}^n \pi_{ij}^j\pi_{jk}^j
    \]
    for all $1\leq i,k\leq n$. Conversely, when these equations hold each path not in $I$ must have a unique maximal vertex, and the concatenation of two paths not in $I$ at their maximal vertices is also not in $I$, for otherwise $\pi_{ik}^n$ would be strictly smaller than the sum on the right. The latter implies that the maximal vertex of any path in the generating set $G$ is external, for otherwise splitting the path at a maximal internal vertex would produce a pair of paths not in $I$ whose concatenation at their maximal vertices was in $I$. This completes the proof.
    \end{proof}

\section{Examples}
\label{sec: examples}

We illustrate our results through five simple examples. We describe each as the category of representations of a quiver  with relations, but the first four can also be interpreted as categories of perverse sheaves.

In each case the quiver will have vertices $1,\ldots,n$ and we order the simple representations accordingly. Following the convention of the previous section, the composite of two arrows $\alpha:i\to j$ and $\beta:j\to k$ is denoted $\alpha\beta$.

The first pair of examples illustrate the fact that a global dimension $2$ category may be highest weight only for some orderings of the simple objects.

\begin{example}
\label{ex:1} Let $A=kQ/I$ where $Q$ is the quiver
\begin{equation}
\label{quiver1}
\begin{tikzcd}
1 \arrow[r,bend left,"\alpha"]  & 2 \arrow[l,bend left,"\beta"]
\end{tikzcd}
\end{equation}
with relations $I=\langle \beta\alpha\rangle$. Its representations are equivalent to the category of perverse sheaves on $\C\PP^1$ stratified by a point and its complement. The global dimension is $2$. This is highest weight by \cite[Thm 3.3.1]{MR1322847}, and indeed this is immediately detected by Green-Schroll's criterion since the maximal vertices in the generating relation $\beta\alpha$ are external.

There are five indecomposable representations:  the Auslander--Reiten quiver is depicted in Figure \ref{fig:AR_quiver_1}. The standard and costandard objects are $\stan{1}=S_1=\costan{1}$,  $\stan{2}=P_2$ and $\costan{2}=I_2$. As expected our numerical criteria hold: $[\stan{1}:S_1] = 1 = [\stan{2}:S_2]$, 
\[
[P_1] = [\stan{1}]+ [\stan{2}] = [\costan{1}:S_1][\stan{1}]+ [\costan{2}:S_1][\stan{2}]
\]
and $[P_2] = [\stan{2}] = [\costan{1}:S_2][\stan{1}]+ [\costan{2}:S_2][\stan{2}]$. As predicted, there is a stability function in which $\stan{1}$ and $\stan{2}$ are stable with $\phi(\stan{1}) < \phi(\stan{2})$ and the Harder--Narasimhan factors of $P_1$ and $P_2$  are standard  --- see Figure \ref{fig:stability_function}.
\end{example}

\begin{figure}
    \begin{tikzcd}
    &&\\
    && P_1=I_1 \ar[dr] &&\\
     & P_2 \ar[ur]\ar[dr]\ar[rr,dash,dashed] && I_2 \ar[dr] &\\
      S_1 \ar[ur]\ar[rr,dash,dashed] && S_2 \ar[ur]\ar[rr,dash,dashed] && S_1 \\
     \ar[r,dash,dashed]& I_1=P_2 \ar[dr]\ar[rr,dash,dashed] && I_2=P_1 \ar[dr] \ar[r,dash,dashed]& {}\\
      S_1 \ar[ur]\ar[rr,dash,dashed] && S_2 \ar[ur]\ar[rr,dash,dashed] && S_1 
    \end{tikzcd}
    
    \caption{Auslander-Reiten quivers of the algebras from Example \ref{ex:1} (top) and Example \ref{ex:3} (bottom).}
    \label{fig:AR_quiver_1}

\bigskip
\bigskip

    \centering
    
    \begin{tikzpicture}
         
         \filldraw[red!20, opacity=0.75] (0,0) -- (0,2) -- (1,3) -- (1,1) -- cycle;
        \draw[blue] (0,0) --(0,2);
        \draw[blue] (0,0) -- (1,1);
        \draw[blue] (0,0) -- (-1,1);

        \draw (-0.05,0) edge (-2.3,0); 
        \draw[dashed] (0.05,0) edge (2.5,0); 
        \draw[fill] (0,0) circle [radius=.04];
        
        \draw[fill, blue] (1,1) circle [radius=.05];
        \node at (1.75,1) {$\scriptstyle{S_1 = \stan{1}}$};
        
        \draw[fill, blue] (-1,1) circle [radius=.05]; 
        \node at (-1.5,1) {$\scriptstyle{S_2}$};
        
        \draw[fill, blue] (0,2) circle [radius=.05];
        \node at (-0.75,2) {$\scriptstyle{P_2=\stan{2}}$};
        
        \draw[fill, red] (1,3) circle [radius=.05];
        \node at (1.5,3) {$\scriptstyle{P_1}$};
        
    \end{tikzpicture}
   \qquad 
    \begin{tikzpicture}
         
         \filldraw[red!20, opacity=0.75] (0,0) -- (0,2) -- (-1,3) -- (-1,1) -- cycle;
        \filldraw[red!20, opacity=0.75] (0,0) -- (0,2) -- (-1,1) -- cycle;
        \draw[blue] (0,0) -- (1,1);
        \draw[blue] (0,0) -- (-1,1);

        \draw (-0.05,0) edge (-2.3,0); 
        \draw[dashed] (0.05,0) edge (2.5,0); 
        \draw[fill] (0,0) circle [radius=.04];
        
        \draw[fill, blue] (-1,1) circle [radius=.05]; 
        \node at (-1.5,1) {$\scriptstyle{S_2}$};
        
        \draw[fill, blue] (1,1) circle [radius=.05];
        \node at (1.75,1) {$\scriptstyle{S_1 = \stan{1}}$};
        
        \draw[fill, red] (0,2) circle [radius=.05];
        \node at (0.5,2) {$\scriptstyle{P_1}$};
        
        \draw[fill, red] (-1,3) circle [radius=.05];
        \node at (-1.75,3) {$\scriptstyle{P_2=\stan{2}}$};
        
    \end{tikzpicture}
    
\bigskip

 \begin{tikzpicture}
         
         \filldraw[red!20, opacity=0.75] (0,0) -- (2,1) -- (4,3) -- (2,2) -- cycle;
         \filldraw[red!20, opacity=0.75] (0,0) -- (2,2) -- (0,4) -- (-2,2) -- cycle;
        \draw[blue] (0,0) --(0,1);
        \draw[blue] (0,0) -- (2,1);
        \draw[blue] (0,0) -- (2,2);
        \draw[blue] (0,0) -- (-2,2);
        \draw[blue] (0,0) -- (-2,1);

        \draw (-0.05,0) edge (-2.3,0); 
        \draw[dashed] (0.05,0) edge (2.5,0); 
        \draw[fill] (0,0) circle [radius=.04];
        
        \draw[fill, blue] (2,1) circle [radius=.05];
        \node at (2.75,1) {$\scriptstyle{S_1 = \stan{1}}$};

        \draw[fill, blue] (0,1) circle [radius=.05];
        \node at (0.5,1) {$\scriptstyle{S_2}$};
        
        \draw[fill, blue] (-2,1) circle [radius=.05]; 
        \node at (-2.5,1) {$\scriptstyle{S_3}$};
        
        \draw[fill, blue] (2,2) circle [radius=.05];
        \node at (2,2.5) {$\scriptstyle{\stan{2}}$};
        
        \draw[fill, red] (4,3) circle [radius=.05];
        \node at (4.5,3) {$\scriptstyle{P_1}$};
        
        \draw[fill, red] (0,4) circle [radius=.05];
        \node at (0,4.5) {$\scriptstyle{P_2}$};

        \draw[fill, blue] (-2,2) circle [radius=.05];
        \node at (-2.75,2) {$\scriptstyle{P_3=\stan{3}}$};
        
    \end{tikzpicture}
    
    \caption{Stability functions for Example \ref{ex:1} (top left), \ref{ex:2} (top right) and (\ref{ex:4}) (bottom). Simple, projective and standard objects are indicated by dots, blue for stable objects and red for unstable. The HN polygons, \ie the convex hulls of the charges of all subobjects, of unstable objects are shaded red. The HN filtration is given by the sequence of vectors of stable objects on the left hand edge of the HN polygon.} 
    \label{fig:stability_function}
\end{figure}

\begin{example}
    \label{ex:2} 
    Consider the quiver (\ref{quiver1}) from the previous example but now with relations $I'=\langle \alpha\beta\rangle$. The algebra $A'=kQ/I'$ is isomorphic to the algebra $A$ in the previous example, but the ordering of the simple objects is reversed. Therefore the Auslander--Reiten quiver is the same  but with the subscripts $1$ and $2$ switched. Now $\stan{2}=P_2$ has $\End{}{\stan{2}} \cong k[x]/x^2$ so $\mod{A'}$ is not highest weight. 
    
    The failure to be highest weight can be detected by noting that:
    \begin{itemize}
        \item the maximal vertex $2$ is internal to the generating relation $\alpha\beta$;
        \item the multiplicity $[\stan{2}:S_2] = 2$;
        \item there is no stability function with $\phi(S_1)<\phi(S_2)$ in which $\stan{2}$ is stable --- see Figure \ref{fig:stability_function}.
    \end{itemize}
\end{example}

The following example illustrates the possibility that the category admits a recollement with factors equivalent to the category of vector spaces, but which is not homological.

\begin{example}
\label{ex:3}
Let $\alg{B}=k Q/J$ where $Q$ is the quiver in (\ref{quiver1}) and $J=\langle \alpha\beta, \beta\alpha \rangle$. Then $\mod{B}$  is equivalent to the category of perverse sheaves on $\C\PP^2$ stratified by $\C\PP^1$ and its complement, see \cite[Example 6.3]{MR0833195}. It has infinite global dimension so is not highest weight. The latter is easily seen using Green--Schroll's criterion as $2$ is a maximal internal vertex of  $\alpha\beta$.

The Auslander--Reiten quiver is depicted in Figure \ref{fig:AR_quiver_1}. The standard and costandard objects are $\stan{1}=S_1=\costan{1}$, $\stan{2}=P_2$ and $\costan{2}=I_2$. Therefore the first numerical condition that $[\stan{1}:S_1] = 1 = [\stan{2}:S_2]$ is met. This corresponds to the fact that the representation category admits a recollement with factors equivalent to the category of vector spaces. However, this recollement is not homological and this is detected by the failure of the second numerical criterion: 
\[
[P_1] \neq [\stan{1}]+[\stan{2}].
\]
Correspondingly, the standard objects are stable for any stability function with $\phi(S_1)< \phi(S_2)$, but the projective $P_1$ has initial Harder--Narasimhan factor $S_2$, which is not standard.
\end{example}

In the next example the representation category is highest weight, but that of the associated monomial algebra is not. This case is therefore not decided by Green-Schroll's criterion.
\begin{example}
\label{ex:4}
Let $C=k Q/I$ where $Q$ is the quiver
\[
\begin{tikzcd}
    1 \arrow[r,bend left,"\alpha"]  & 2 \arrow[l,bend left,"\beta"] \arrow[r,bend left,"\gamma"]& 3 \arrow[l,bend left,"\delta"]
\end{tikzcd}
\]
and $I=\langle \alpha\gamma, \delta\beta, \delta\gamma,  \beta\alpha-\gamma\delta \rangle$. Then $\mod{C}$ is equivalent to the category of perverse sheaves on $\C\PP^2$ stratified by $\C\PP^0 \subset \C\PP^1 \subset \C \PP^2$. This is highest weight by \cite[Thm 3.3.1]{MR1322847} and has global dimension $4$.  The associated monomial algebra --- see \cite[\S 2]{MR4017924} for the construction --- is given by the ideal
$I_\text{mon} = \langle \alpha\beta\alpha,\beta\alpha\beta, \alpha\gamma, \delta\beta, \delta\gamma, \gamma\delta \rangle$. 
Applying Green-Schroll's criterion to this, or simply noting that $kQ/I_\text{mon}$ has infinite global dimension, shows that $kQ/I_\text{mon}$ is not quasi-hereditary. 

The projective, standard, injective and costandard objects of $\mod{C}$ are indicated in the table below. The numerical criteria showing that the representation category is highest weight can be verified from this data.

\medskip
\begin{center}
    
  \begin{tabular}{|l|c|c|c|c|}
    \hline
      & Projective & Standard & Injective & Costandard \\
    \hline
       $i=1$  
       & $\scriptstyle{\begin{matrix}1 \\ \ \ 2 \\1 \end{matrix}}$ 
       & $\scriptstyle{\begin{matrix} 1\\ \\ {} \end{matrix}}$
       & $\scriptstyle{\begin{matrix} 1 \\ \ \  2 \\1 \end{matrix}}$
       & $\scriptstyle{\begin{matrix} \\ \\ 1 \end{matrix}}$
       \\
       \hline
       $i=2$  
       & $\scriptstyle{\begin{matrix} 2 \\ 1\ 3 \\2 \end{matrix}}$ 
       & $\scriptstyle{\begin{matrix} 2 \\1\ \  \\ {} \end{matrix}}$
       & $\scriptstyle{\begin{matrix} 2 \\ 1 \ 3 \\2 \end{matrix}}$
       & $\scriptstyle{\begin{matrix} \\ 1\ \  \\ 2 \end{matrix}}$
       \\
        \hline
       $i=3$  
       & $\scriptstyle{\begin{matrix}  \ \ 3 \\ 2 \end{matrix}}$ 
       & $\scriptstyle{\begin{matrix}  \ \ 3 \\ 2\end{matrix}}$
       & $\scriptstyle{\begin{matrix} 2 \\  \ \ 3  \end{matrix}}$
       & $\scriptstyle{\begin{matrix} 2 \\  \ \ 3 \end{matrix}}$
       \\
       \hline
    \end{tabular}
\end{center}

\medskip

See Figure \ref{fig:stability_function} for a depiction of a stability function in which the standard objects are stable of strictly increasing phase and the Harder--Narasimhan factors of the projective objects are standard.
\end{example}

The final example is one of a properly stratified category which is not highest weight. Here there is a stability function in which the proper standard objects are stable, but no stability function in which the standard objects are so.
\begin{example}
\label{ex: 5}
\cite[Example 2.3]{MR1749881}
Let $D=kQ/I$ where $Q$ is the quiver
\[
\begin{tikzcd}
    1 \arrow[r,bend left,"\alpha"]  & 2 \arrow[l,bend left,"\beta"] 
    \ar[loop, in=330,out=30,looseness=5,"\gamma"]
    \end{tikzcd}
\]
and $I=\langle \alpha\gamma, \beta\alpha, \gamma^2\rangle$. The standard, proper standard and projective objects are as in the table below. The category $\mod{D}$ is properly stratified by inspection, but not highest weight because $P_1$ has no filtration by standard objects.

\medskip
\begin{center}
    
  \begin{tabular}{|l|c|c|c|}
    \hline
      & Projective & Standard & Proper standard \\
    \hline
       $i=1$  
       & $\scriptstyle{\begin{matrix}1 \\ 2 \\1 \end{matrix}}$ 
       & $\scriptstyle{\begin{matrix} 1\\ \\ {} \end{matrix}}$
       & $\scriptstyle{\begin{matrix} 1\\ \\ {} \end{matrix}}$
        \\
       \hline
       $i=2$  
       & $\scriptstyle{\begin{matrix} 2 \\ 1\ 2 \\ \ \ 1 \end{matrix}}$ 
       & $\scriptstyle{\begin{matrix} 2 \\ 1\ 2 \\ \ \ 1 \end{matrix}}$ 
       & $\scriptstyle{\begin{matrix} 2 \\1 \\ {} \end{matrix}}$
       \\
        \hline
    \end{tabular}
\end{center}
\medskip

The standard object $\stan{2}$ is semistable when $\phi(S_1) \leq \phi(S_2)$ but never stable because it is not a brick. Applying Theorem~\ref{thm:new hwc characterisations} this confirms $\mod{D}$ is not highest weight. The proper standard $\pstan{2}$ is stable whenever $\phi(S_1)<\phi(S_2)$, and in that case the Harder--Narasimhan filtrations of the projective objects coincide with their filtrations by proper standard objects as predicted by Theorem~\ref{thm:properly stratified}.

\end{example}

\bibliographystyle{alpha}
\bibliography{References.bib}

\end{document}